\documentclass{article}%
\usepackage{amsmath}
\usepackage{amsfonts}
\usepackage{amssymb}
\usepackage{theorem}
\usepackage{graphicx}%
\providecommand{\U}[1]{\protect \rule{.1in}{.1in}}
\newtheorem{theorem}{Theorem}

\newtheorem{lemma}[theorem]{Lemma}

\newtheorem{proposition}[theorem]{Proposition}
{\theorembodyfont{\rmfamily}

}

\newenvironment{proof}[1][Proof]{\noindent \textbf{#1.} }{\  \rule{0.5em}{0.5em}\vskip0.3truecm}
\begin{document}

\title{Representative functions of maximal monotone operators and bifunctions}
\author{M. Bianchi\thanks{Universit{\`a} Cattolica del Sacro Cuore, Milano,
Italy\texttt{(monica.bianchi@unicatt.it)}}, N. Hadjisavvas\thanks{King Fahd
University of Petroleum and Minerals, Kingdom of Saudi
Arabia\texttt{(nhadjisavvas@kfupm.edu.sa)}}, R. Pini\thanks{Universit{\`a}
degli Studi di Milano--Bicocca, Italy\texttt{(rita.pini@unimib.it)}}}
\maketitle

\begin{abstract}
\noindent The aim of this paper is to show that every representative function
of a maximal monotone operator is the Fitzpatrick transform of a bifunction
corresponding to the operator. In this way we exhibit the relation between the
recent theory of representative functions, and the much older theory of saddle
functions initiated by Rockafellar.

\end{abstract}



\section{Introduction}

Given a maximal monotone operator $T$ in a Banach space $X$, a class
$\mathcal{H}(T)$ of convex, lower semicontinuous functions on the product
space $X\times X^{\ast}$ was introduced by Fitzpatrick \cite{fitz}, that
represent $T$ in the following sense: each function $\varphi \in \mathcal{H}(T)$
determines exactly the graph of $T$ as the set of coincidence of $\varphi$
with the usual duality product. The class $\mathcal{H}(T)$ has a minimum
element, the so-called Fitzpatrick function. The theory of representative
functions has proven to be very fruitful, and has lead to major advances in
the theory of maximal monotone operators.

On the other hand, to every maximal monotone operator corresponds a class of
bifunctions defined on the product $X\times X$. It had been shown that
bifunctions, apart from being an interesting object of study in themseleves,
especially in relation with equilibrium problems, are also useful for the
study of maximal monotone operators. Actually, to every monotone operator
corresponds a class of bifunctions such that, in some sense, the operator is
the subdifferential of the bifunctions (see \cite{Iusem} for details). To each
such bifunction, one defines its Fitzpatrick transform \cite{Bo-Gr}. It has
been shown that the Fitzpatrick transform of every bifunction corresponding to
a maximal monotone operator, is a representative function of the operator
\cite{AH2}. One of the aims of the present paper, it to answer the following
question: Does every representative function of a maximal monotone operator
arise in this way? In other words, given a representative function, does there
exist a bifunction corresponding to the operator, such that its Fitzpatrick
transform is the given representative function? As we will see, the answer is
yes, and in fact one may find all such bifunctions. In addition, these
bifunctions may be chosen to be \textquotedblleft saddle functions", i.e.,
concave in the first variable and convex in the second one. Our results
establish a close connection between the recent theory of representative
functions and the much older theory of saddle functions by Rockafellar
\cite{Rock-70b, Rock-saddle}, Krauss \cite{Krauss-Roum, Krauss-nonlin} etc. In
fact, some of our results are not really new; what is new is their connection
with the theory of maximal monotone operators and the Fitzpatrick function.

\section{Preliminaries}

Let $X$ be a real Banach space and $X^{\ast}$ its topological dual. Denote by
$\pi$ the duality product $\pi(x,x^{\ast})=\langle x^{\ast},x\rangle.$ We will
use the weak$^{\ast}$ topology in $X^{\ast}$, so its dual with respect to this
topology is $X$. The space $X\times X^{\ast}$ is endowed with the product
topology, so its dual is $X^{\ast}\times X$ with the canonical duality pairing
defined by
\[
\left \langle (x^{\ast},x),(y,y^{\ast})\right \rangle =\langle x^{\ast}%
,y\rangle+\langle y^{\ast},x\rangle.
\]
Given a subset $K$ of $X$ we will denote by $\mathrm{co}K$ and $\overline
{\mathrm{{co}}}K$ its convex hull and closed convex hull, respectively;
moreover, we will denote by $\delta_{K}$ the indicator function of $K,$ i.e.,
\[
\delta_{K}(x):=%
\begin{cases}
+\infty & \text{if $x\notin K$,}\\
0 & \text{if $x\in K$.}%
\end{cases}
\]
In the following we will denote by $\overline{\mathbb{R}}$ the set
$\mathbb{R}\cup \{-\infty,+\infty \}$.

\subsection{Some elements of convex analysis}

In the sequel we recall some definitions according to \cite{Rock-70b}; it
should be noted that some definitions (such as closedness) differ from
definitions found in other sources. Given a function $f:X\rightarrow
\overline{\mathbb{R}}$, its domain and epigraph are, respectively, the sets
$\operatorname*{dom}f=\{x\in X:f(x)<+\infty \}$ and $\operatorname*{epi}%
f=\{(x,\mu)\in X\times \mathbb{R}:f(x)\leq \mu \}$. The function $f$ is called
convex if $\operatorname*{epi}f$ is convex. The convex hull $\mathrm{co}\,f$
of a function $f$ is the function which is the greatest convex minorant of
$f$. Equivalently,%
\begin{align*}
\mathrm{co}\,f(x)  &  =\inf \{ \mu:(x,\mu)\in \mathrm{co}(\operatorname*{epi}%
f)\} \\
&  =\inf \{ \sum_{i=1}^{m}\lambda_{i}f(x_{i}):\sum_{i=1}^{m}\lambda_{i}%
x_{i}=x,\;x_{i}\in \operatorname*{dom}f,\; \sum_{i=1}^{m}\lambda_{i}=1,\;
\lambda_{i}\geq0\}.
\end{align*}

If $f$ is convex, its \textit{closure} $\overline{f}$ is defined as the
pointwise supremum of all continuous affine functions majorized by $f$:%
\[
\overline{f}=\sup \{h:h\text{ is continuous affine, }h\leq f\}.
\]

If $f$ is convex and never takes the value $-\infty$, its closure
$\overline{f}$ is the greatest lower semicontinuous (lsc) convex minorant of
$f$; it is the function whose epigraph is the closure of $\operatorname*{epi}%
f$. However, if $f$ is convex and $f(x)=-\infty$ for some $x,$ then
$\overline{f}\equiv-\infty$. A convex function is said to be \textit{closed}
if $\overline{f}=f$. A convex function $f:X\rightarrow \overline{\mathbb{R}}$
is called \textit{proper} if $f(x)>-\infty,$ for any $x\in X$, and it is not
identically equal to $+\infty$. For a proper convex function, closedness is
the same as lower semicontinuity. For every function $f$, we denote by
$\overline{\mathrm{co}}f$ the function $\overline{\mathrm{co}f}.$

The \textit{convex conjugate} $f^{\ast}:X^{\ast}\rightarrow \overline
{\mathbb{R}}$ of a function $f:X\rightarrow \overline{\mathbb{R}}$ is given by
\[
f^{\ast}(x^{\ast}):=\sup_{x\in X}\{ \langle x^{\ast},x\rangle-f(x)\}.
\]
The function $f^{\ast}$ is closed and convex, and it is proper if and only if
$f$ is proper. Moreover, $(\overline{\mathrm{co}}f)^{\ast}=\left(
\overline{f}\right)  ^{\ast}=f^{\ast}.$ In this paper, the convex conjugate of
a function $g:X^{\ast}\rightarrow \overline{\mathbb{R}}$ will be meant to be
defined in $X$ rather than $X^{\ast \ast}$. For every function $f$,
$f^{\ast \ast}=\overline{\mathrm{co}}f$.

For any function $f:X\rightarrow \overline{\mathbb{R}},$ the well-known Fenchel
inequality holds:
\begin{equation}
f^{\ast}(x^{\ast})\geq \langle x^{\ast},x\rangle-f(x)\text{ for all }x\in
X,x^{\ast}\in X^{\ast}. \label{Fenchel_inequality}%
\end{equation}

A function $f:X\rightarrow \overline{\mathbb{R}}$ is called concave if $-f$ is
convex. Given a function $f$, its concave hull $\mathrm{cv}f$ is the function
$\mathrm{cv}f=-\mathrm{co}(-f)$, i.e. the smallest concave majorant of $f$.
Equivalently,%
\[
\mathrm{cv}f(x)=\sup \{ \sum_{i=1}^{m}\lambda_{i}f(x_{i}):\sum_{i=1}^{m}%
\lambda_{i}x_{i}=x,\;f(x_{i})>-\infty,\; \sum_{i=1}^{m}\lambda_{i}=1,\;
\lambda_{i}\geq0\}.
\]

If $f$ is concave, its closure is by definition the function $\overline
{f}=-\overline{(-f)}$. In this case,%
\[
\overline{f}=\inf \{h:h\text{ is continuous affine, }h\geq f\}.
\]

\subsection{Monotone operators and representative functions}

Given a multivalued operator $T:X\rightrightarrows X^{\ast}$, we recall that
its domain and graph are, respectively, the sets $D(T)=\{x\in X:\,T(x)\neq
\emptyset \}$ and $\operatorname*{gph}T=\{(x,x^{\ast})\in X\times X^{\ast
}:x^{\ast}\in T(x)\}$.

In the sequel, we will assume that $D(T)\neq \emptyset$.

The multivalued operator $T$ is called \textit{monotone} if for any $x,y\in
D(T)$ the inequality $\langle x^{\ast}-y^{\ast},x-y\rangle \geq0$ holds
whenever $x^{\ast}\in T(x)$ and $y^{\ast}\in T(y)$. In particular, the
monotone operator $T$ is called \textit{maximal} if its graph is not properly
included in the graph of any other monotone operator.

We recall that if $T$ is maximal monotone and $X^{\ast}$ is reflexive, then
$\overline{D(T)}$ is convex, so $\overline{\operatorname*{co}}D(T)=\overline
{D(T)}$ \cite{Rock-conv}.

Given a multivalued operator $T$, the class $\mathcal{H}(T)$ of
\textit{representative functions} of $T$ is defined as the class of all closed
and convex functions $\varphi:X\times X^{\ast}\rightarrow \overline{\mathbb{R}%
}$ such that:
\[%
\begin{cases}
\varphi(x,x^{\ast})\geq \langle x^{\ast},x\rangle,\text{ for all }(x,x^{\ast
})\in X\times X^{\ast}\\
(x,x^{\ast})\in \mathrm{gph}T\Rightarrow \varphi(x,x^{\ast})=\langle x^{\ast
},x\rangle
\end{cases}
\]
Since we assume that $\mathrm{gph}T\neq \emptyset,$ then any representative
function is proper and thus, closedness is equivalent to lsc. With respect to
each of its variables, $\varphi$ might be improper, but it is still convex and closed.

To any operator $T:X\rightrightarrows X^{\ast}$, one associates its
\emph{Fitzpatrick function} \cite{fitz} $\mathcal{F}_{T}:X\times X^{\ast
}\rightarrow \mathbb{R}\cup \{+\infty \}$ defined by
\begin{align*}
\mathcal{F}_{T}(x,x^{\ast})  &  =\; \sup_{(y,y^{\ast})\in \operatorname*{gph}%
T}(\langle y^{\ast}-x^{\ast},x-y\rangle+\langle x^{\ast},x\rangle)\\
&  =\; \sup_{(y,y^{\ast})\in \operatorname*{gph}T}\left(  \left \langle x^{\ast
},y\right \rangle +\left \langle y^{\ast},x-y\right \rangle \right)  .
\end{align*}
The Fitzpatrick function $\mathcal{F}_{T}$ is convex and lsc with respect to
the pair $(x,x^{\ast})$. For any maximal monotone operator $T,$ the function
$\mathcal{F}_{T}$ belongs to $\mathcal{H}(T)$, and is in fact the smallest
function of this family. In addition, for every $\varphi \in \mathcal{H}(T)$,
the equality $\varphi(x,x^{\ast})=\langle x^{\ast},x\rangle$ characterizes the
points in the graph of $T$.

On the other hand, the function $\sigma_{T}:X\times X^{\ast}\rightarrow
\overline{\mathbb{R}}$ defined by
\[
\sigma_{T}(x,x^{\ast}):=\overline{\mathrm{co}}\left(  \pi+\delta
_{\mathrm{gph}T}\right)  (x,x^{\ast})
\]
is the greatest representative function in $\mathcal{H}(T),$ if $T$ is maximal
monotone \cite{BuSva}.

The function $\sigma_{T}$ is connected to the Fitzpatrick function via the
following equalities:%
\[
\mathcal{F}_{T}(x,x^{\ast})=\sigma_{T}^{\ast}(x^{\ast},x),\qquad
\mathcal{F}_{T}^{\ast}(x^{\ast},x)=\sigma_{T}(x,x^{\ast})
\]
(see for instance \cite{BuSva, Mle-Sv}).

In case of maximal monotone operators, the transpose of the conjugate of any
representative function $\varphi$, i.e. the function $(\varphi^{*})^{t}$
defined by $(\varphi^{*})^{t}(x,x^{*})=\varphi^{*}(x^{*},x),$ where
\[
\varphi^{\ast}(x^{\ast},x)=\sup_{(y,y^{\ast})\in X\times X^{\ast}}(\langle
x^{\ast},y\rangle+\langle y^{\ast},x\rangle-\varphi(x,x^{\ast})),
\]
is also a representative function of $T$ \cite{BuSva}.

Given a representative function $\varphi$, its domain is a subset of $X\times
X^{\ast}$. We will denote by $P_{1}\operatorname*{dom}\varphi$ the projection
of $\operatorname*{dom}\varphi$ on $X$, i.e.,
\[
P_{1}\operatorname*{dom}\varphi=\{x\in X:\exists x^{\ast}\in X^{\ast}\text{
such that }\varphi(x,x^{\ast})<+\infty \}.
\]

\begin{proposition}
\label{domain(T)} Let $\varphi$ be a representative function of some operator
$T$ and let $x\in X$ be given.

\begin{enumerate}
\item[a.] $\operatorname*{co}D(T)\subseteq P_{1}\mathrm{dom}\varphi$.

\item[b.] If $T$ is maximal monotone, then $P_{1}\mathrm{dom}\varphi
\subseteq \overline{\operatorname*{co}}D(T)$. If in addition
$\operatorname*{int}\operatorname*{co}D(T)\neq \emptyset$, then
$\operatorname*{int}D(T)=\operatorname*{int}P_{1}\mathrm{dom}\varphi$.
\end{enumerate}
\end{proposition}

\begin{proof}
a. Let $x\in D(T)$. Then there exists $x^{\ast}\in T(x)$, thus $\varphi
(x,x^{\ast})=\left \langle x^{\ast},x\right \rangle \in \mathbb{R}$. Hence
$D(T)\subseteq P_{1}\mathrm{dom}\varphi$. Since $P_{1}\mathrm{dom}\varphi$ is
the projection of a convex set, it is convex, thus the inclusion
$\operatorname*{co}D(T)\subseteq P_{1}\mathrm{dom}\varphi$ follows.

b. Let $x\in P_{1}\mathrm{dom}\varphi$. Then there exists $x^{\ast}\in
X^{\ast}$ such that $\varphi(x,x^{\ast})\in \mathbb{R}$. Assume that
$x\notin \overline{\operatorname*{co}}D(T)$; then there exist $\varepsilon>0$
and $v^{\ast}\in X^{\ast}$ such that $\left \langle v^{\ast},x-y\right \rangle
>\varepsilon$ for all $y\in D(T)$. We can choose $v^{\ast}$ so that%
\[
\left \langle v^{\ast},x-y\right \rangle \geq \varphi(x,x^{\ast})-\left \langle
x^{\ast},x\right \rangle ,\quad \forall y\in D(T).
\]

Since the Fitzpatrick function $\mathcal{F}_{T}$ is the minimum element of the
class of representative functions, for all $(y,y^{\ast})\in \operatorname{gph}%
T$ we obtain%
\[
\left \langle x^{\ast}-y^{\ast},y-x\right \rangle +\left \langle x^{\ast
},x\right \rangle \leq \mathcal{F}_{T}(x,x^{\ast})\leq \varphi(x,x^{\ast}).
\]

It follows that%
\[
\left \langle (x^{\ast}+v^{\ast})-y^{\ast},x-y\right \rangle \geq0,\quad
\forall(y,y^{\ast})\in \operatorname{gph}T.
\]

Since $T$ is maximal monotone, $x^{\ast}+v^{\ast}\in T(x)$, contradicting
$x\notin \overline{\operatorname*{co}}D(T)$.

To show the equality of the interiors, we remark that $\operatorname*{co}%
D(T)\subseteq P_{1}\mathrm{dom}\varphi \subseteq \overline{\operatorname*{co}%
}D(T)$ implies that $\operatorname*{int}D(T)\subseteq \operatorname*{int}%
\operatorname*{co}D(T)\subseteq \operatorname*{int}P_{1}\mathrm{dom}%
\varphi \subseteq \operatorname*{int}\overline{\operatorname*{co}}D(T)$.
If\linebreak$\operatorname*{int}\operatorname*{co}D(T)\neq \emptyset$, then
$\operatorname*{int}\operatorname*{co}D(T)=\operatorname*{int}\overline
{\operatorname*{co}}D(T)$. In addition, it is known that $\operatorname*{int}%
D(T)=\operatorname*{int}\operatorname*{co}D(T)$ \cite{Rock-local}, so we
obtain $\operatorname*{int}D(T)=\operatorname*{int}P_{1}\mathrm{dom}\varphi$.
\end{proof}

See also \cite{SimZa} for the inclusion $\operatorname*{co}D(T)\subseteq
P_{1}\mathrm{dom}\mathcal{F}_{T}$, and \cite[Theorem 2.2]{Sim} for the
equality $\operatorname*{int}D(T)=\operatorname*{int}P_{1}\mathrm{dom}%
\mathcal{F}_{T}$.

Note that in general $\operatorname*{co}D(T)\neq P_{1}\mathrm{dom}\varphi
\neq \overline{\operatorname*{co}}D(T)$, as seen in the following example. Let
$T:(0,1)\rightarrow \mathbb{R}$ be a continuous increasing function such that
$T(x)=\frac{1}{1-x}$ near $1$ and $T(x)=-\frac{1}{x^{2}}$ near $0$. Then $T$
is maximal monotone, and for every $x^{\ast}\geq0$,%
\[
\mathcal{F}_{T}(1,x^{\ast})=\sup_{y\in(0,1)}(T(y)+x^{\ast}y-yT(y))\leq
\sup_{y\in(0,1)}T(y)(1-y)+x^{\ast}<+\infty
\]
while for every $x^{\ast}\in \mathbb{R}$,%
\[
\mathcal{F}_{T}(0,x^{\ast})=\sup_{y\in(0,1)}(x^{\ast}y-yT(y))=+\infty.
\]

Hence $D(T)\neq P_{1}\mathrm{dom}\mathcal{F}_{T}=(0,1]\neq \overline
{\operatorname*{co}}D(T)$.

\subsection{Bifunctions and saddle functions}

\label{sec_bifunction}

By the term \emph{bifunction} we understand any function $F:X\times
X\rightarrow{\overline{\mathbb{R}}}$. A bifunction $F\ $is said to be
\emph{normal} if there exists a nonempty set $C\subseteq X$ such that
$F(x,y)=-\infty$ if and only if $x\notin C.$ The set $C$ will be called the
\emph{domain} of $F$ and denoted by $D(F)$. In particular, if $F$ is normal,
then $F$ is not identically $-\infty$.

The bifunction $F$ is said to be monotone if
\[
F(x,y)\leq-F(y,x)
\]
for all $x,y\in X$. Every monotone bifunction satisfies the inequality
$F(x,x)\leq0$, for all $x\in X$.

Given a bifunction $F$, we define the operator $A^{F}:X\rightrightarrows
X^{\ast}$ by
\[
A^{F}(x)=\{x^{\ast}\in X^{\ast}:F(x,y)\geq \langle x^{\ast},y-x\rangle,\;
\forall y\in X\}.
\]
Note that, if $F$ is normal, then $D(A^{F})\subseteq D(F)$, and
\begin{equation}
F(x,x)\geq0\quad \forall x\in D(A^{F}). \label{Fge0}%
\end{equation}
It is easy to check that the operator $A^{F}$ is monotone whenever $F$ is a
monotone bifunction; moreover, $F(x,x)=0$ for all $x\in D(A^{F})$. The
converse is not true: $A^{F}$ may be monotone while $F$ is not. See
\cite{HJML} for examples, and Proposition \ref{Fmon} below.

On the other hand, given an operator $T$ one can define the bifunction
$G_{T}:X\times X\rightarrow{\overline{\mathbb{R}}}$ by
\begin{equation}
G_{T}(x,y)=\sup_{x^{\ast}\in T(x)}\langle x^{\ast},y-x\rangle. \label{G_T}%
\end{equation}
The bifunction $G_{T}$ is normal and $D(G_{T})=D(T)$; furthermore
$G_{T}(x,x)=0$ for all $x\in D(T),$ and $G_{T}(x,\cdot)$ is closed and convex
for all $x\in X$. If $T$ is a monotone operator, then $G_{T}$ is a monotone bifunction.

We can associate to each bifunction $F$ its \textit{Fitzpatrick transform}
\begin{equation}
\varphi_{F}(x,x^{\ast})=\sup_{y\in X}\left(  \langle x^{\ast},y\rangle
+F(y,x)\right)  =(-F(\cdot,x))^{\ast}(x^{\ast}), \label{Fitzpatrick-transform}%
\end{equation}
i.e., $\varphi_{F}$ is the conjugate of $-F$ with respect to its first
variable (see, for instance, \cite{AH2} and \cite{Bo-Gr}).

If $F(y,\cdot)$ is lsc and convex for all $y\in X,$ then $\varphi_{F}$ is also
lsc and convex on $X\times X^{\ast}$. Moreover, if $F$ is normal, then
\[
\varphi_{F}(x,x^{\ast})=\sup_{y\in X}\left(  \langle x^{\ast},y\rangle
+F(y,x)\right)  =\sup_{y\in D(F)}\left(  \langle x^{\ast},y\rangle
+F(y,x)\right)  ;
\]
this implies that $\varphi_{F}(x,x^{\ast})>-\infty$ for all $(x,x^{\ast})\in
X\times X^{\ast}$, and $\varphi_{F}$ is closed.

Note that, for any operator $T$, the following equality holds:
\begin{align}
\mathcal{F}_{T}(x,x^{\ast})  &  =\sup_{(y,y^{\ast})\in \operatorname*{gph}%
T}\left(  \left \langle x^{\ast},y\right \rangle +\left \langle y^{\ast
},x-y\right \rangle \right) \nonumber \\
&  =\sup_{y\in X}\left(  \langle x^{\ast},y\rangle+\sup_{y^{\ast}\in
T(y)}\langle y^{\ast},x-y\rangle \right) \nonumber \\
&  =\sup_{y\in D(T)}\left(  \langle x^{\ast},y\rangle+G_{T}(y,x)\right)
\nonumber \\
&  =\varphi_{G_{T}}(x,x^{\ast}). \label{FT}%
\end{align}

Given a bifunction $F,$ one can associate to $F$ also the \emph{upper
Fitzpatrick transform} $\varphi^{F}$ given by
\begin{equation}
\varphi^{F}(x,x^{\ast})=\sup_{y\in X}\left(  \left \langle x^{\ast
},y\right \rangle -F(x,y)\right)  =F(x,\cdot)^{\ast}(x^{\ast})\text{,}
\label{Fitzp-upper}%
\end{equation}
and the operator $\,^{F}\!A,$ given by
\[
\,^{F}\!A(x)=\{x^{\ast}\in X^{\ast}:\,-F(y,x)\geq \langle x^{\ast}%
,y-x\rangle,\quad \forall y\in X\}
\]
(see for instance \cite{AH2}, \cite{Bo-Gr}).

A class of bifunctions widely used in mathematical literature is the class of
\emph{saddle} functions, i.e., bifunctions which are concave in the first
argument, and convex in the second one (see, for instance, \cite{Rock-70b}).
For these functions, let us recall some basic definitions.

One denotes by $\mathrm{cl}_{2}F$ the bifunction obtained by closing
$F(x,\cdot)$ as a convex function, for every $x\in X;$ likewise, one denotes
by $\mathrm{cl}_{1}F$ the bifunction obtained by closing $F(\cdot,y)$ as a
concave function, for every $y\in X.$

Two saddle functions $F,H$ are called equivalent if $\mathrm{cl}%
_{i}F=\mathrm{cl}_{i}H$, $i=1,2$; in this case we write $F\sim H$. Clearly,
$\sim$ is an equivalence relation. A saddle function $F$ is called
\emph{closed} if $\mathrm{cl}_{1}F\sim \mathrm{cl}_{2}F\sim F$. It is called
\emph{lower closed} if $\mathrm{cl}_{2}\mathrm{cl}_{1}F=F$, and \emph{upper
closed} if $\mathrm{cl}_{1}\mathrm{cl}_{2}F=F$. It is easy to see that every
lower closed and every upper closed saddle function is closed. Also, if $F\sim
H$ and $F$ is closed, then $H$ is closed too.

Given a saddle function $F$, we define following \cite{Krauss-Roum}
\[
\mathrm{dom}_{1}F=\{x\in X:\, \mathrm{cl}_{2}F(x,y)>-\infty,\quad \forall y\in
X\},
\]
and
\[
\mathrm{dom}_{2}F=\{y\in X:\, \mathrm{cl}_{1}F(x,y)<+\infty,\quad \forall x\in
X\}.
\]
Note that, if $F$ is a saddle function such that $\mathrm{cl}_{2}F=F$ and $F$
is not identically $-\infty,$ then $F$ is normal, and $\mathrm{dom}%
_{1}F=D(F).$ Moreover, if $F$ is a saddle function, such that $\mathrm{cl}%
_{1}F=F$ and $F$ is not identically $+\infty,$ then $(x,y)\mapsto
-F(y,x)=\hat{F}(x,y)$ is normal, and $\mathrm{dom}_{2}F=D(\hat{F}).$

The next proposition shows that the quantities $\varphi_{F}$, $\varphi^{F}$,
$A^{F}$ and $^{F}\!A$ depend only on the equivalent class to which the saddle
function $F$ belongs.

\begin{proposition}
\label{Prop-ex-rem}Two saddle functions $F$ and $H$ are equivalent if and only
if $\varphi_{F}=\varphi_{H}$ and $\varphi^{F}=\varphi^{H}$. In addition, if
$F$ and $H$ are equivalent then $A^{H}=A^{F}$ and $^{H}\!A=\,^{F}\!A$.
\end{proposition}

\begin{proof}
If $F$ is a saddle function, then
\begin{equation}
A^{F}=A^{\mathrm{cl}_{2}F},\quad \varphi_{F}=\varphi_{\mathrm{cl}_{1}F}.
\label{cl_2F}%
\end{equation}
Here, the first equality stems from the definition of $A^{F}$ and the closure
of a convex function, while the second one is a consequence of relation
(\ref{Fitzpatrick-transform}) and the fact that $f^{\ast}=\left(  \overline
{f}\right)  ^{\ast}$ for every convex function $f$.

In a similar way as in \eqref{cl_2F}, if $F$ is a saddle function, then
\begin{equation}
\,^{F}\!A=\,^{\mathrm{cl}_{1}F}\!A,\quad \varphi^{F}=\varphi^{\mathrm{cl}_{2}%
F}. \label{cl_bis}%
\end{equation}
It follows from the above relations that whenever $F$ and $H$ are equivalent
saddle functions, then $A^{H}=A^{F}$, $^{H}\!A=^{F}\! \!A$, $\varphi
_{H}=\varphi_{F}$ and $\varphi^{H}=\varphi^{F}$.

Now assume that $F$ and $H$ are two saddle functions such that $\varphi
_{H}=\varphi_{F}$ and $\varphi^{H}=\varphi^{F}$. From the first equality we
deduce that
\[
(-F(\cdot,x))^{\ast}(x^{\ast})=(-H(\cdot,x))^{\ast}(x^{\ast}),
\]
and, taking again the Fenchel conjugate, we get that $\mathrm{cl}%
_{1}F=\mathrm{cl}_{1}H.$ Moreover, from $\varphi^{H}=\varphi^{F},$ we get
that
\[
F(x,\cdot)^{\ast}(x^{\ast})=H(x,\cdot)^{\ast}(x^{\ast}),
\]
and, by taking the conjugates, we obtain that $\mathrm{cl}_{2}F=\mathrm{cl}%
_{2}H.$ Thus, $F$ and $H$ are equivalent.
\end{proof}

\section{The class of representative functions and saddle functions}

Given a maximal monotone operator $T$, there is a whole family of
representative functions $\mathcal{H}(T)$, one of which is its Fitzpatrick function.

In this section we will address the following question: given a maximal
monotone operator $T$ and one of its representative functions $\varphi
\in \mathcal{H}(T)$, is it true that $\varphi$ arises as the Fitzpatrick
transform of a bifunction related to $T$? The answer is positive, and in
addition the bifunction can be chosen to be a closed saddle function, as we
will see in the sequel.

In the following proposition, we will show that, under suitable assumptions,
both the Fitzpatrick transform and the upper Fitzpatrick transform of a
bifunction $F$ belong to $\mathcal{H}(T).$
%
We prove first a lemma:

\begin{lemma}
\label{lemma_3} \label{Lemma}Assume that $T$ is a maximal monotone operator
and $F:X\times X\rightarrow \overline{\mathbb{R}}$ is a bifunction such that
$T(x)\subseteq A^{F}(x)\cap \,^{F}\!A(x)$ for all $x\in X$.

$(i)$ If $F(x,\cdot)$ is lsc and convex for all $x\in X$, then $\varphi_{F}%
\in \mathcal{H}(T)$.

$(ii)$ If $F(\cdot,y)$ is usc and concave for all $y\in X$, then $\varphi
^{F}\in \mathcal{H}(T)$.
\end{lemma}

\begin{proof}
$(i)$ Since $F(x,\cdot)$ is lsc and convex for all $x\in X$, $\varphi_{F}$ is
lsc and convex. Assume first that, for some $(x,x^{\ast})\in X\times X^{\ast
},$
\[
\varphi_{F}(x,x^{\ast})\leq \left \langle x^{\ast},x\right \rangle .
\]
For every $(y,y^{\ast})\in \mathrm{gph}(T)$, using successively that
$T(x)\subseteq A^{F}(x)$ and the definition of $\varphi_{F}$,%
\begin{equation}
\left \langle y^{\ast},x-y\right \rangle +\left \langle x^{\ast},y\right \rangle
\leq \left \langle x^{\ast},y\right \rangle +F(y,x)\leq \varphi_{F}(x,x^{\ast
})\leq \left \langle x^{\ast},x\right \rangle . \label{phiF}%
\end{equation}

Hence, $\left \langle y^{\ast}-x^{\ast},y-x\right \rangle \geq0$ for all
$(y,y^{\ast})\in \mathrm{gph}(T)$, so, by the maximality of $T$, $x^{\ast}\in
T(x)$. Putting $y=x$ and $y^{\ast}=x^{\ast}$ in (\ref{phiF}) we deduce that
$\varphi_{F}(x,x^{\ast})=\left \langle x^{\ast},x\right \rangle $. It follows
that $\varphi_{F}(x,x^{\ast})<\left \langle x^{\ast},x\right \rangle $ is not
possible, hence $\varphi_{F}(x,x^{\ast})\geq \left \langle x^{\ast
},x\right \rangle $ for all $(x,x^{\ast})\in X\times X^{\ast}$.

Now, if
\[
\varphi_{F}(x,x^{\ast}) > \left \langle x^{\ast},x\right \rangle ,
\]
by contradiction it is easy to show that $(x,x^{\ast})\notin \mathrm{gph}(T)$.
Indeed, if $(x,x^{\ast})\in \mathrm{gph}(T),$ from $T(x)\subseteq \,^{F}\!A(x)$
we deduce that, for all $y\in X$,%
\[
F(y,x)+\left \langle x^{\ast},y\right \rangle \leq \left \langle x^{\ast
},x\right \rangle .
\]

By taking the supremum for all $y\in X$ we get that $\varphi_{F}(x,x^{\ast
})\leq \left \langle x^{\ast},x\right \rangle $, a contradiction. Thus, by the
first part of the proof, $\varphi_{F}(x,x^{\ast}) \ge \left \langle x^{\ast
},x\right \rangle $ for all $(x,x^{\ast}) \in X\times X^{\ast}$, and
$\varphi_{F}(x,x^{\ast}) = \left \langle x^{\ast},x\right \rangle $ if and only
if $(x,x^{\ast})\in \mathrm{gph}(T)$, i.e. $\varphi_{F}\in \mathcal{H}(T)$.

$(ii)$ We apply part $(i)$ to the bifunction $\hat{F}(x,y):=-F(y,x)$. We note
that $\hat{F}(x,\cdot)$ is lsc and convex for all $x\in X$, while
$\varphi_{\hat{F}}=\varphi^{F}$, $A^{\hat{F}}(x)=\,^{F}\!A(x)$ and $^{\hat{F}%
}\!A(x)=A^{F}(x)$. We deduce that $\varphi^{F}=\varphi_{\hat{F}}\in
\mathcal{H}(T)$.
\end{proof}

Note that in the above lemma we do not assume that $F$ is monotone. In the
special case of a monotone bifunction $F$, one has $A^{F}(x)\subseteq
\,^{F}\!A(x)$, so the assumption $T(x)\subseteq A^{F}(x)\cap \,^{F}\!A(x)$ is
equivalent to $T(x)=A^{F}(x)$ in view of the maximality of $T$.

\begin{proposition}
\label{Prop_is_repr}Assume that $T$ is a maximal monotone operator and
$F:X\times X\rightarrow \overline{\mathbb{R}}$ is a closed saddle function. If
$T(x)\subseteq A^{F}(x)\cap \,^{F}\!A(x)$ for all $x\in X$, then $\varphi
_{F}\in \mathcal{H}(T)$ and $\varphi^{F}\in \mathcal{H}(T)$.
\end{proposition}

\begin{proof}
Since $F$ is closed, $F\sim \mathrm{cl}_{2}F$. By Proposition \ref{Prop-ex-rem}%
, $A^{\mathrm{cl}_{2}F}=A^{F}$, $^{\mathrm{cl}_{2}F}\!A=\,^{F}\!A$ and
$\varphi_{\mathrm{cl}_{2}F}=\varphi_{F}$. By applying part $(i)$ of the Lemma
\ref{lemma_3} to $\mathrm{cl}_{2}F$, we conclude that $\varphi_{F}%
=\varphi_{\mathrm{cl}_{2}F}\in \mathcal{H}(T)$. Likewise, using $F\sim
\mathrm{cl}_{1}F$ and part $(ii)$ of the lemma, we obtain $\varphi^{F}%
\in \mathcal{H}(T)$.
\end{proof}

In the main result of this section, we prove that all representative functions
of $T$ can be realized by taking the Fitzpatrick transform of suitable saddle functions.

In what follows, $\varphi^{\ast}$ will be the convex conjugate of $\varphi$
with respect to the pair of variables $(x,x^{\ast})$, while expressions like
$\left(  \varphi^{\ast}(\cdot,x)\right)  ^{\ast}(y)$ will mean the convex
conjugate of $\varphi^{\ast}$ (in $X$) with respect to the variable $x^{\ast}$ only.

Given $\varphi \in \mathcal{H}(T)$, define the bifunction $F$ by the formula%
\begin{equation}
F(x,y)=\sup_{x^{\ast}\in X^{\ast}}\{ \left \langle x^{\ast},y\right \rangle
-\varphi^{\ast}(x^{\ast},x)\}=\left(  \varphi^{\ast}(\cdot,x)\right)  ^{\ast
}(y) \label{F}%
\end{equation}

By taking the second conjugate in (\ref{F}) with respect to $y$ we also find%
\begin{equation}
\varphi^{\ast}(x^{\ast},x)=\left(  F(x,\cdot)\right)  ^{\ast}(x^{\ast}%
)=\sup_{y\in X}\{ \left \langle x^{\ast},y\right \rangle -F(x,y)\}.
\label{phiF3}%
\end{equation}

\begin{theorem}
\label{Theor-basic}Let $T$ be a maximal monotone operator and $\varphi
\in \mathcal{H}(T).$ Then the bifunction $F$ defined by the formula (\ref{F})
has the following properties:

\begin{enumerate}
\item[(a)] $F$ is a saddle function such that $\mathrm{cl}_{2}F=F$.

\item[(b)] $F$ is normal, with $\operatorname*{co}D(T)\subseteq D(F)\subseteq
\overline{\operatorname*{co}}D(T)$;

\item[(c)] $A^{F}=\,^{F}\!A=T$;

\item[(d)] $\varphi_{F}=\varphi$ and $\varphi^{F}=(\varphi^{\ast})^{t}$.
\end{enumerate}
\end{theorem}

\begin{proof}
$(a)$ For every $x\in X$, $F(x,\cdot)$ is the Fenchel transform of a function,
therefore it is closed and convex. In addition, for every $y\in X$, $F(x,y)$
is the supremum over $x^{\ast}$ of a family of functions which are concave
with respect to the pair $(x,x^{\ast})$; hence $F(\cdot,y)$ is concave.

$(b)$ Since $F(x,\cdot)$ is convex and closed, if $F(x,y_{0})=-\infty$ for
some $(x,y_{0})$, then $F(x,\cdot)=-\infty;$ in particular, $F$ is normal. In
addition, it is evident that $x\in D(F)$ if and only if $\varphi^{\ast
}(x^{\ast},x)<+\infty$ for some $x^{\ast}\in X^{\ast}$, i.e., $D(F)=P_{1}%
\mathrm{dom}(\varphi^{\ast})^{t}$. Since $(\varphi^{\ast})^{t}$ is a
representative function of $T$, the inclusions then follow from Proposition
\ref{domain(T)}.

$(c)$ Let us assume that $x^{\ast}\in T(x)$. Taking into account that
$\varphi \in \mathcal{H}(T)$ entails that $(\varphi^{\ast})^{t}\in
\mathcal{H}(T)$ too, for all $y\in X$ we find%
\[
F(x,y)=\sup_{z^{\ast}\in X^{\ast}}\{ \left \langle z^{\ast},y\right \rangle
-\varphi^{\ast}(z^{\ast},x)\} \geq \left \langle x^{\ast},y\right \rangle
-\varphi^{\ast}(x^{\ast},x)=\left \langle x^{\ast},y-x\right \rangle ;
\]
hence, $T(x)\subseteq A^{F}(x)$.

Assume now that $x^{\ast}\in A^{F}(x)$. Then, taking into account
(\ref{phiF3}), we find successively
\begin{align*}
\left \langle x^{\ast},y-x\right \rangle \leq F(x,y),\, \forall y\in X  &
\Leftrightarrow \sup_{y\in X}\{ \left \langle x^{\ast},y\right \rangle -F(x,y)\}
\leq \left \langle x^{\ast},x\right \rangle \\
&  \Leftrightarrow \varphi^{\ast}(x^{\ast},x)\leq \left \langle x^{\ast
},x\right \rangle .
\end{align*}
Using again that $(\varphi^{\ast})^{t}$ is a representative function, we find
that $x^{\ast}\in T(x)$ so $A^{F}=T$.

Assume that $x^{\ast}\in \,^{F}\!A(x)$. This is equivalent to%
\[
\forall y\in X,\quad \left \langle x^{\ast},y-x\right \rangle +F(y,x)\leq0
\]
i.e.,%
\begin{equation}
\forall y\in X,\forall y^{\ast}\in X^{\ast},\quad \left \langle x^{\ast
},y-x\right \rangle +\left \langle y^{\ast},x\right \rangle -\varphi^{\ast
}(y^{\ast},y)\leq0. \label{FA}%
\end{equation}

Since $(\varphi^{\ast})^{t}$ is also a representative function, if we take
$(y,y^{\ast})\in \operatorname*{gph}T$ then $\varphi^{\ast}(y^{\ast
},y)=\left \langle y^{\ast},y\right \rangle $ so we deduce from (\ref{FA}) that%
\[
\forall(y,y^{\ast})\in \operatorname*{gph}T,\quad \left \langle y^{\ast}-x^{\ast
},y-x\right \rangle \geq0.
\]

From the maximality of $T$ we deduce that $x^{\ast}\in T(x)$. Conversely, if
$x^{\ast}\in T(x)$, then for every $(y,y^{\ast})\in X\times X^{\ast}$ we find,
using that $\mathcal{F}_{T}$ is the smallest representative function:%
\[
\varphi^{\ast}(y,y^{\ast})\geq \mathcal{F}_{T}(y,y^{\ast})\geq \left \langle
x^{\ast},y\right \rangle +\left \langle y^{\ast},x\right \rangle -\left \langle
x^{\ast},x\right \rangle
\]
so (\ref{FA}) holds. Hence $x^{\ast}\in \,^{F}\!A(x)$.

$(d)$ Since $\varphi$ is proper, lsc and convex, $\varphi^{\ast \ast}=\varphi.$
We have from (\ref{phiF3}), using also that $F(x,\cdot)$ is convex and
closed,
\begin{align*}
\varphi(x,x^{\ast})  &  =\sup_{(y^{\ast},y)\in X^{\ast}\times X}\left(
\langle y^{\ast},x\rangle+\langle x^{\ast},y\rangle-\varphi^{\ast}(y^{\ast
},y)\right) \\
&  =\sup_{(y^{\ast},y)\in X^{\ast}\times X}\left(  \langle y^{\ast}%
,x\rangle+\langle x^{\ast},y\rangle-(F(y,\cdot))^{\ast}(y^{\ast})\right) \\
&  =\sup_{y\in Y}(\langle x^{\ast},y\rangle+\sup_{y^{\ast}\in X^{\ast}}\left(
\langle y^{\ast},x\rangle-(F(y,\cdot))^{\ast}(y^{\ast})\right) \\
&  =\sup_{y\in Y}\left(  \langle x^{\ast},y\rangle+(F(y,\cdot))^{\ast \ast
}(x)\right) \\
&  =\sup_{y\in Y}\left(  \langle x^{\ast},y\rangle+F(y,x)\right) \\
&  =\varphi_{F}(x,x^{\ast}).
\end{align*}

Finally, comparing (\ref{phiF3}) and (\ref{Fitzp-upper}) we get immediately
$\varphi^{F}=(\varphi^{\ast})^{t}$.
\end{proof}

Note that $F$ is not monotone in general:

\begin{proposition}
\label{Fmon}The bifunction $F$ defined by (\ref{F}) is monotone if and only if
$\varphi(x,x^{\ast})\leq \varphi^{\ast}(x^{\ast},x),$ for every $x\in X,$
$x^{\ast}\in X^{\ast}$.
\end{proposition}

\begin{proof}
In order to see when $F$ is monotone, notice that the condition $F(y,x)\leq
-F(x,y)$ is equivalent to%
\[
\left \langle y^{\ast},x\right \rangle -\varphi^{\ast}(y^{\ast},y)\leq
-\left \langle x^{\ast},y\right \rangle +\varphi^{\ast}(x^{\ast},x)\text{,
}\quad \forall x,y\in X\text{, }x^{\ast},y^{\ast}\in X^{\ast},
\]
or, alternatively,
\[
\sup_{y\in X,y^{\ast}\in X^{\ast}}\left(  \left \langle x^{\ast},y\right \rangle
+\left \langle y^{\ast},x\right \rangle -\varphi^{\ast}(y^{\ast},y)\right)
\leq \varphi^{\ast}(x^{\ast},x)\text{, }\quad \forall x\in X,x^{\ast}\in
X^{\ast},
\]
i.e.,
\[
\varphi(x,x^{\ast})\leq \varphi^{\ast}(x^{\ast},x)\text{, }\quad \forall x\in
X,x^{\ast}\in X^{\ast},
\]
since $\varphi^{\ast \ast}=\varphi$.
\end{proof}

The bifunction $F$ defined by (\ref{F}) is not the only saddle function that
satisfies $(c)$ and $(d)$ of Theorem \ref{Theor-basic}. According to
Proposition \ref{Prop-ex-rem}, any saddle function equivalent to $F$ also
satisfies these conditions. An example of a saddle function equivalent to $F$
is given by
\begin{equation}
\widetilde{F}(x,y)=-\sup_{y^{\ast}\in X^{\ast}}\{ \left \langle y^{\ast
},x\right \rangle -\varphi(y,y^{\ast})\}=-(\varphi(y,\cdot))^{\ast}(x)
\label{F-hat_form}%
\end{equation}
Indeed the next proposition holds:

\begin{proposition}
\label{F-closed}The bifunction $\widetilde{F}$ is a saddle function and
satisfies
\begin{equation}
\widetilde{F}=\mathrm{cl}_{1}F,\qquad F=\mathrm{cl}_{2}\widetilde{F}.
\label{F-Fhat}%
\end{equation}
Consequently, $F$ is lower closed, $\widetilde{F}$ is upper closed, and
$F\sim \widetilde{F}$. Finally,%
\[
F(x,y)\leq \widetilde{F}(x,y),\qquad \forall(x,y)\in X\times X.
\]

\end{proposition}

\begin{proof}
The proof that $\widetilde{F}$ is a saddle function is similar to the proof of
the analogous assertion for $F$ in Theorem \ref{Theor-basic}$(a)$. By Theorem
\ref{Theor-basic} and relation (\ref{Fitzpatrick-transform}),
\[
\varphi(y,y^{\ast})=\varphi_{F}(y,y^{\ast})=\left(  -F(\cdot,y)\right)
^{\ast}(y^{\ast}),
\]
and therefore $\widetilde{F}$ is also given by the formula
\begin{equation}
-\widetilde{F}(x,y)=\left(  -F(\cdot,y)\right)  ^{\ast \ast}(x) \label{F-hat}%
\end{equation}
i.e., $\widetilde{F}=\mathrm{cl}_{1}F.$ In addition, in view of
\eqref{F-hat_form},
\begin{align*}
\varphi^{\ast}(x^{\ast},x)  &  =\sup_{(y,y^{\ast})\in X\times X^{\ast}}\{
\left \langle x^{\ast},y\right \rangle +\left \langle y^{\ast},x\right \rangle
-\varphi(y,y^{\ast})\} \\
&  =\sup_{y\in X}\left \{  \left \langle x^{\ast},y\right \rangle +\sup_{y^{\ast
}\in X^{\ast}}\left(  \left \langle y^{\ast},x\right \rangle -\varphi(y,y^{\ast
})\right)  \right \} \\
&  =\sup_{y\in X}\left \{  \left \langle x^{\ast},y\right \rangle -\widetilde
{F}(x,y)\right \}  =\left(  \widetilde{F}(x,\cdot)\right)  ^{\ast}(x^{\ast}).
\end{align*}
Therefore,
\[
F(x,y)=\left(  \varphi^{\ast}(\cdot,x\right)  )^{\ast}(y)=\left(
\widetilde{F}(x,\cdot)\right)  ^{\ast \ast}(y)=\mathrm{cl}_{2}\widetilde
{F}(x,y).
\]

The inequality $F\leq \widetilde{F}$ follows from $F=\mathrm{cl}_{2}%
\widetilde{F}$.

The remaining assertions of the proposition are immediate consequences of
equalities (\ref{F-Fhat}).
\end{proof}

The next proposition summarizes some results about $\widetilde{F}$, similar to
Theorem \ref{Theor-basic}.

\begin{proposition}
Let $T$ be a maximal monotone operator, $\varphi \in \mathcal{H}(T)$ and
$\widetilde{F}$ be defined by (\ref{F-hat_form}). Then:

\begin{enumerate}
\item[(a)] $\widetilde{F}$ is a saddle function such that $\mathrm{cl}%
_{1}\widetilde{F}=\widetilde{F}.$

\item[(b)] $-\widetilde{F}^{t}$ is normal, and $\mathrm{co}D(T)\subseteq
D(-\widetilde{F}^{t})\subseteq \overline{\mathrm{co}}D(T),$ where
$\widetilde{F}^{t}(x,y)=\widetilde{F}(y,x);$

\item[(c)] $\varphi_{\widetilde{F}}=\varphi$ and $\varphi^{\widetilde{F}%
}=(\varphi^{\ast})^{t}$;

\item[(d)] $T=A^{\widetilde{F}}=\,^{\widetilde{F}}\!A$.
\end{enumerate}
\end{proposition}

\begin{proof}
Parts $(a)$, $(c)$ and $(d)$ follow from Propositions \ref{Prop-ex-rem} and
\ref{F-closed}. The proof of $(b)$ follows the same steps as the proof of
Theorem \ref{Theor-basic}$(b)$.
\end{proof}

In the next result, we prove that the set of all saddle functions that are
equivalent to $F$ is exactly the set of saddle functions between $F$ and
$\widetilde{F}$. Consequently, the bifunctions $F$ and $\widetilde{F}$ play
the role of maximal and minimal element in the class of saddle functions
satisfying the equalities%
\begin{equation}
\varphi_{H}=\varphi,\quad \varphi^{H}=(\varphi^{\ast})^{t}.\label{Aphi}%
\end{equation}

\begin{proposition}
Let $H$ be a saddle function. Then $H$ satisfies \eqref{Aphi} if and only if
$F\leq H\leq \widetilde{F}.$
\end{proposition}

\begin{proof}
It is easy to see that every saddle function $H$ such that $F\leq
H\leq \widetilde{F}$ is equivalent to $F$ (because $\mathrm{cl}_{1}%
F\leq \mathrm{cl}_{1}H\leq \mathrm{cl}_{1}\widetilde{F}=\mathrm{cl}_{1}F,$ and
the same for $\mathrm{cl}_{2}$), hence it satisfies (\ref{Aphi}). Conversely,
if $H$ satisfies \eqref{Aphi}, then $H\sim F$. From the equalities
\[
\mathrm{cl}_{1}H=\mathrm{cl}_{1}{F}=\widetilde{F},\quad \mathrm{cl}%
_{2}H=\mathrm{cl}_{2}{F}={F},
\]
and since for every convex (concave) function the convex (concave) closure is
smaller (greater) than the function, we get that $F\leq H\leq \widetilde{F}.$
\end{proof}

We conclude by illustrating the particular case where $\varphi=\mathcal{F}%
_{T}.$ We will construct the saddle functions $F$ and $\tilde{F},$ whose
existence is part of Theorem \ref{Theor-basic} and Proposition \ref{F-closed},
and we will show how they are related to $G_{T}.$

In view of \eqref{Fitzpatrick-transform} and (\ref{FT}),
\begin{equation}
\mathcal{F}_{T}(x,x^{\ast})=\varphi_{G_{T}}(x,x^{\ast})=\left(  -G_{T}%
(\cdot,x)\right)  ^{\ast}(x^{\ast}). \label{FG}%
\end{equation}
Since the bifunction $G_{T}$ is not saddle, in general, let us consider the
bifunction $\hat{G}_{T}:X\times X\rightarrow \overline{\mathbb{R}}$ defined by
$\hat{G}_{T}(\cdot,y)=\mathrm{cv\,}G_{T}(\cdot,y)$, for each $y\in X$ (see
also \cite{Krauss-nonlin, AH2}).

Since $G_{T}(x,y)>-\infty$ is equivalent to $x\in D(T)$, $\hat{G}_{T}$ is
given by

\smallskip \noindent%
\[
\hat{G}_{T}(x,y):=\sup \{ \sum_{i=1}^{k}\alpha_{i}G_{T}(x_{i},y):\,x=\sum
_{i=1}^{k}\alpha_{i}x_{i},\,x_{i}\in D(T),\sum_{i=1}^{k}\alpha_{i}=1,\,
\alpha_{i}\geq0\}.
\]

By construction, $\hat{G}_{T}(\cdot,y)$ is concave; also, $\hat{G}_{T}%
(x,\cdot)$ is convex and closed, as a supremum of convex and closed functions.
Thus, $\hat{G}_{T}$ is a saddle function such that $\mathrm{cl}_{2}\hat{G}%
_{T}=\hat{G}_{T}$.

Since $T$ is monotone, we know that $G_{T}$ is monotone, thus%
\[
G_{T}(x,y)\leq-G_{T}(y,x),\qquad \forall(x,y)\in X\times X.
\]

If we take the convex hull with respect to $y$ of both sides we find%
\[
G_{T}(x,y)\leq-\hat{G}_{T}(y,x),\qquad \forall(x,y)\in X\times X.
\]

Now we take the concave hull with respect to $x$ of both sides and we deduce%
\[
\hat{G}_{T}(x,y)\leq-\hat{G}_{T}(y,x),\qquad \forall(x,y)\in X\times X.
\]

Consequently, $\hat{G}_{T}$ is monotone.

We have
\begin{align*}
\left(  \varphi_{G_{T}}\right)  ^{\ast}(x^{\ast},x)  &  =\sup_{(y,y^{\ast})\in
X\times X^{\ast}}\{ \left \langle x^{\ast},y\right \rangle +\left \langle
y^{\ast},x\right \rangle -\varphi_{G_{T}}(y,y^{\ast})\} \\
&  =\sup_{(y,y^{\ast})\in X\times X^{\ast}}\{ \left \langle x^{\ast
},y\right \rangle +\left \langle y^{\ast},x\right \rangle -\left(  -G_{T}%
(\cdot,y)\right)  ^{\ast}(y^{\ast})\} \\
&  =\sup_{y\in X}\left \{  \left \langle x^{\ast},y\right \rangle +\left(
-G_{T}(\cdot,y)\right)  ^{\ast \ast}(x)\right \} \\
&  =\sup_{y\in X}\left \{  \left \langle x^{\ast},y\right \rangle -\mathrm{cl}%
_{1}\hat{G}_{T}(x,y)\right \} \\
&  =\left(  \mathrm{cl}_{1}\hat{G}_{T}(x,\cdot)\right)  ^{\ast}(x^{\ast})
\end{align*}
thus%
\[
F(x,y)=\left(  \left(  \varphi_{G_{T}}\right)  ^{\ast}(\cdot,x)\right)
^{\ast}(y)=\left(  \mathrm{cl}_{1}\hat{G}_{T}(x,\cdot)\right)  ^{\ast \ast
}(y)=\mathrm{cl}_{2}\mathrm{cl}_{1}\hat{G}_{T}(x,y).
\]

That is, $F$ is the \textquotedblleft lower closure\textquotedblright \ of
$\hat{G}_{T}$ \cite{Krauss-Roum}. Note that by Proposition \ref{Fmon}, $F$ is
monotone, because $\mathcal{F}_{T}(x,x^{\ast})\leq \sigma_{T}(x^{\ast},x).$

Since $\hat{G}_{T}$ is convex and closed in the second variable,
$\mathrm{cl}_{2}\hat{G}_{T}=\hat{G}_{T}$. Using that for every saddle function
$H$ the saddle function $\mathrm{cl}_{1}\mathrm{cl}_{2}H$ is upper closed
\cite{Rock-conv, Krauss-Roum} we find
\[
\widetilde{F}=\mathrm{cl}_{1}F=\mathrm{cl}_{1}\mathrm{cl}_{2}\mathrm{cl}%
_{1}\hat{G}_{T}=\mathrm{cl}_{1}\mathrm{cl}_{2}\mathrm{cl}_{1}\mathrm{cl}%
_{2}\hat{G}_{T}=\mathrm{cl}_{1}\mathrm{cl}_{2}\hat{G}_{T}=\mathrm{cl}_{1}%
\hat{G}_{T}.
\]
Thus, $\widetilde{F}$ is the \textquotedblleft upper closure" of $\hat{G}_{T}$.

\end{document}